\documentclass[10pt]{article}
\font\smallit=cmti10
\font\smalltt=cmtt10

\usepackage{amssymb,latexsym,amsmath,epsfig,amsthm} 

\makeatletter

\renewcommand\section{\@startsection {section}{1}{\z@}
{-30pt \@plus -1ex \@minus -.2ex}
{2.3ex \@plus.2ex}
{\normalfont\normalsize\bfseries}}

\renewcommand\subsection{\@startsection{subsection}{2}{\z@}
{-3.25ex\@plus -1ex \@minus -.2ex}
{1.5ex \@plus .2ex}
{\normalfont\normalsize\bfseries}}

\renewcommand{\@seccntformat}[1]{\csname the#1\endcsname. }

\makeatother

\newtheorem{lemma}{Lemma}[section]

\newtheorem{prop}[lemma]{Proposition}
\newtheorem{thm}[lemma]{Theorem}
\newtheorem{cor}[lemma]{Corollary}
\theoremstyle{definition}

\theoremstyle{remark}
\newtheorem{remark}[lemma]{Remark}

\newcommand{\e}{\varepsilon}

\newcommand{\card}{\text{card}}

\newcommand{\B}{\mathcal B}

\renewcommand{\epsilon}{\varepsilon}

\begin{document}

\begin{center}
\uppercase{\bf Expansions in non-integer bases: lower order revisited}
\vskip 20pt
{\bf Simon Baker}\\
{\smallit School of Mathematics, University of Manchester,  Manchester, United Kingdom}\\
{\tt simonbaker412@gmail.com}\\
\vskip 10pt
{\bf Nikita Sidorov}\\
{\smallit School of Mathematics, University of Manchester,  Manchester, United Kingdom}\\
{\tt sidorov@manchester.ac.uk}\\
\end{center}
\vskip 30pt

\centerline{\smallit Received: , Revised: , Accepted: , Published: } 
\vskip 30pt

\centerline{\bf Abstract}

\noindent
Let $q\in(1,2)$ and $x\in[0,\frac1{q-1}]$. We say that a sequence $(\e_i)_{i=1}^{\infty}\in\{0,1\}^{\mathbb{N}}$ is an expansion of $x$ in base~$q$ (or a $q$-expansion) if
\[
x=\sum_{i=1}^{\infty}\e_iq^{-i}.
\]
For any $k\in\mathbb N$, let $\B_k$ denote the set of $q$ such that there exists $x$ with exactly $k$ expansions in base $q$. In \cite{Sid1} it was shown that $\min\B_2=q_2\approx 1.71064$, the appropriate root of $x^{4}=2x^{2}+x+1$. In this paper we show that for any $k\geq 3$, $\min\B_k=q_f\approx1.75488$, the appropriate root of $x^3=2x^2-x+1$.

\pagestyle{myheadings}
\markright{\smalltt INTEGERS: 14 (2014)\hfill}
\thispagestyle{empty}
\baselineskip=12.875pt
\vskip 30pt

\section{Introduction}

Let ${q}\in(1,2)$ and $I_{q}=[0,\frac{1}{{q}-1}]$. Given $x\in \mathbb{R}$, we say that a sequence $(\epsilon_{i})_{i=1}^{\infty}\in\{0,1\}^{\mathbb{N}}$ is a \textit{${q}$-expansion} for $x$ if
\begin{equation}
\label{beta equation}
x=\sum_{i=1}^{\infty}\frac{\epsilon_{i}}{{q}^{i}}.
\end{equation} Expansions in non-integer bases were pioneered in the papers of R\'enyi \cite{Renyi} and Parry \cite{Parry}.

It is a simple exercise to show that $x$ has a ${q}$-expansion if and only if $x\in I_{q}$, when $(\ref{beta equation})$ holds we will adopt the notation $x=(\epsilon_{1},\epsilon_{2},\ldots)_{q}$.  Given $x\in I_{q}$ we denote the set of ${q}$-expansions for $x$ by $\Sigma_{q}(x)$, i.e., $$\Sigma_{q}(x)=\Big\{(\epsilon_{i})_{i=1}^{\infty}\in \{0,1\}^{\mathbb{N}} : \sum_{i=1}^{\infty}\frac{\epsilon_{i}}{{q}^{i}}=x\Big\}.$$

In \cite{Erdos} it is shown that for ${q}\in(1,\frac{1+\sqrt{5}}{2})$ the set $\Sigma_{q}(x)$ is uncountable for all $x\in (0,\frac{1}{{q}-1})$; the endpoints of $I_{q}$ trivially have a unique ${q}$-expansion for all ${q}\in(1,2)$. In \cite{SidVer} it is shown that for ${q}=\frac{1+\sqrt{5}}{2}$ every $x\in(0,\frac{1}{{q}-1})$ has uncountably many ${q}$-expansions unless $x=\frac{(1+\sqrt{5})n}{2}\bmod1$, for some $n\in\mathbb{Z}$, in which case $\Sigma_{q}(x)$ is infinite countable. Moreover, in \cite{DaKa} it is shown that for all ${q}\in(\frac{1+\sqrt{5}}{2},2)$ there exists $x\in (0,\frac{1}{{q}-1})$ with a unique ${q}$-expansion. In this paper we will be interested in the set of ${q}\in(1,2)$ for which there exists $x\in I_{q}$ with precisely $k$ ${q}$-expansions. More specifically, we will be interested in the set
$$
\B_{k}:=\Big\{{q}\in(1,2)| \textrm{ there exists } x\in \Big(0,\frac{1}{{q}-1}\Big) \textrm{ satisfying } \#\Sigma_{q}(x)=k\Big\}.
$$
It was shown in \cite{EJ} that $\B_k\neq\varnothing$ for any $k\ge2$. Similarly we can define $\B_{\aleph_{0}}$ and $\B_{2^{\aleph_{0}}}$. The reader should bear in mind the possibility that the number of expansions could lie strictly between countable infinite and the continuum. By the above remarks it is clear that $\B_{1}=(\frac{1+\sqrt{5}}{2},2)$. In \cite{Sid1} the following theorem was shown to hold.

\begin{thm}
\label{Nik's thm}
\begin{itemize}
	\item The smallest element of $\B_{2}$ is
$$
{q}_{2}\approx 1.71064,
$$
the appropriate root of $x^{4}=2x^{2}+x+1$.
	\item The next smallest element of $\B_{2}$ is
$$
{q}_{f}\approx 1.75488,
$$
the appropriate root of $x^{3}=2x^{2}-x+1$.
	\item For each $k\in\mathbb{N}$ there exists $\gamma_{k}>0$ such that $(2-\gamma_{k},2)\subset \B_{j}$ for all $1\leq j\leq k$.
\end{itemize}
\end{thm}

The following theorem is the central result of the present paper. It answers a question posed by V.~Komornik \cite{Kom} (see also \cite[Section~5]{Sid1}).

\begin{thm}
\label{Main thm}
For $k\geq 3$ the smallest element of $\B_{k}$ is ${q}_{f}$.
\end{thm}

The range of ${q}>\frac{1+\sqrt5}2$  which are ``sufficiently close'' to the golden ratio is referred to in \cite{Sid1} as the {\em lower order}, which explains the title of the present paper.

During our proof of Theorem~\ref{Main thm} we will also show that ${q}_{f}\in \B_{\aleph_{0}}$, which combined with our earlier remarks, Theorem~\ref{Nik's thm}, Theorem~\ref{Main thm} and a result in \cite{Sid} which states that for ${q}\in[\frac{1+\sqrt{5}}{2},2)$ almost every $x\in I_{q}$ has a continuum of ${q}$-expansions, we can conclude the following.

\begin{thm}
\label{all cardinalities}
In base $q_{f}$ all situations occur: there exist $x\in I_{q}$ having exactly $k$
$q$-expansions for each $k=1,2,\ldots$, $k=\aleph_0$ or $k=2^{\aleph_0}$. Moreover, $q_{f}$ is the smallest $q\in(1,2)$ satisfying this property.
\end{thm}

Before proving Theorem~\ref{Main thm} it is necessary to recall some theory. In what follows we fix $T_{{q},0}(x)={q} x$ and $T_{{q},1}(x)={q} x -1$, we will typically denote an element of $\bigcup_{n=0}^{\infty}\{T_{{q},0},T_{{q},1}\}^{n}$ by $a;$ here $\{T_{{q},0},T_{{q},1}\}^{0}$ denotes the set consisting of the identity map. Moreover, if $a= (a_{1},\ldots ,a_{n})$ we shall use $a(x)$ to denote $(a_{n}\circ \ldots  \circ a_{1})(x)$ and $|a|$ to denote the length of $a$.

We let $$\Omega_{q}(x)=\Big\{(a_{i})_{i=1}^{\infty}\in \{T_{{q},0},T_{{q},1}\}^{\mathbb{N}}:(a_{n}\circ \ldots \circ a_{1})(x)\in I_{q}
 \textrm{ for all } n\in\mathbb{N}\Big\}.$$
The significance of $\Omega_{q}(x)$ is made clear by the following lemma.

\begin{lemma}
\label{Bijection lemma}
$\#\Sigma_{q}(x)=\#\Omega_{q}(x)$ where our bijection identifies $(\epsilon_{i})_{i=1}^{\infty}$ with $(T_{{q},\epsilon_{i}})_{i=1}^{\infty}$.
	
\end{lemma}
The proof of Lemma~\ref{Bijection lemma} is contained within \cite{Baker}. It is an immediate consequence of Lemma~\ref{Bijection lemma} that we can interpret Theorem~\ref{Main thm} in terms of $\Omega_{q}(x)$ rather than $\Sigma_{q}(x)$.

An element $x\in I_{q}$ satisfies $T_{{q},0}(x)\in I_{q}$ and $T_{{q},1}(x)\in I_{q}$ if and only if $x\in[\frac{1}{q},\frac{1}{{q}({q}-1)}]$. Moreover, if $\#\Sigma_{q}(x)>1$ or equivalently $\#\Omega_{q}(x)>1$, then there exists a unique minimal sequence of transformations $a$ such that $a(x)\in[\frac{1}{q},\frac{1}{{q}({q}-1)}]$. In what follows we let $S_{q}:=[\frac{1}{q},\frac{1}{{q}({q}-1)}]$. The set $S_q$ is usually referred to as the {\em switch region}. We will also make regular use of the fact that if $x\in I_{q}$ and $a$ is a sequence of transformations such that $a(x)\in I_{q}$, then
\begin{equation}
\label{Cardinality inequality}
\#\Omega_{q}(x)\geq \#\Omega_{q}(a(x)) \textrm{ or equivalently }\#\Sigma_{q}(x)\geq \#\Sigma_{q}(a(x)),
\end{equation}
this is immediate from the definition of $\Omega_{q}(x)$ and Lemma~\ref{Bijection lemma}.

In the course of our proof of Theorem~\ref{Main thm} we will frequently switch between $\Sigma_{q}(x)$ and the dynamical interpretation of $\Sigma_{q}(x)$ provided by Lemma~\ref{Bijection lemma}, often considering $\Omega_{q}(x)$ will help our exposition.

The following lemma is a consequence of \cite[Theorem~2]{GlenSid}.

\begin{lemma}
\label{Unique expansions lemma}
Let ${q}\in(\frac{1+\sqrt{5}}{2},{q}_{f}]$, if $x\in I_{q}$ has a unique ${q}$-expansion $(\epsilon_{i})_{i=1}^{\infty}$, then $$(\epsilon_{i})_{i=1}^{\infty}\in \Big\{0^{k}(10)^{\infty},1^{k}(10)^{\infty},0^{\infty},1^{\infty}\Big\},$$ where $k\ge0$. Similarly, if $(\epsilon_{i})_{i=1}^{\infty}\in \{0^{k}(10)^{\infty},1^{k}(10)^{\infty},0^{\infty},1^{\infty}\}$ then for ${q}\in(\frac{1+\sqrt{5}}{2},2)$ $x=((\epsilon_{i})_{i=1}^{\infty})_{q}$ has a unique ${q}$-expansion given by $(\epsilon_{i})_{i=1}^{\infty}$.
\end{lemma}

In Lemma~\ref{Unique expansions lemma} we have adopted the notation $(\epsilon_{1}\dots\epsilon_{n})^{k}$ to denote the concatenation of $(\epsilon_{1}\dots \epsilon_{n})\in \{0,1\}^{n}$ by itself $k$ times and $(\epsilon_{1}\dots \epsilon_{n})^{\infty}$ to denote the infinite sequence obtained by concatenating $\epsilon_{1}\dots \epsilon_{n}$ by itself infinitely many times, and we will use this notation throughout.

The following lemma follows from the branching argument first introduced in \cite{Sid2}.
\begin{lemma}
\label{Branching lemma dynamical}
Let $k\geq 2$, $x\in I_{q}$ and suppose $\#\Sigma_{q}(x)=k$ or equivalently $\#\Omega_{q}(x)=k$. If $a$ is the unique minimal sequence of transformations such that $a(x)\in S_{q}$, then
$$
\#\Omega_{q}(T_{{q},1}(a(x)))+\#\Omega_{q}(T_{{q},0}(a(x)))=k.
$$
Moreover, $1\leq \#\Omega_{q}(T_{{q},1}(a(x)))<k$ and $1\leq \#\Omega_{q}(T_{{q},0}(a(x)))<k$.
\end{lemma}

The following result is an immediate consequence of Lemma~\ref{Bijection lemma} and Lemma~\ref{Branching lemma dynamical}.

\begin{cor}
\label{subset lemma}
$\B_{k}\subset \B_{2}$ for all $k\geq 3$.
\end{cor}

An outline of our proof of Theorem~\ref{Main thm} is as follows: first of all we will show that ${q}_{f}\in \B_{k}$ for all $k\geq 1$, then by Theorem~\ref{Nik's thm} and Corollary~\ref{subset lemma} to prove Theorem~\ref{Main thm} it suffices to show that ${q}_{2}\notin \B_{k}$ for all $k\geq 3$. But by an application of Lemma~\ref{Branching lemma dynamical} to show that ${q}_{2}\notin \B_{k}$ for all $k\geq 3$ it suffices to show that ${q}_{2}\notin \B_{3}$ and ${q}_{2}\notin \B_{4}$.

\section{Proof that ${q}_{f}\in \B_{k}$ for all $k\geq 1$}

To show that ${q}_{f}\in \B_{k}$ for all $k\geq 1$, we construct an $x\in I_{{q}_{f}}$ satisfying $\#\Sigma_{{q}_{f}}(x)=k$ explicitly.

\begin{prop}
\label{k expansions prop}
For each $k\geq 1$ the number $x_{k}=(1(0000)^{k-1}0(10)^{\infty})_{{q}_{f}}$ satisfies $\#\Sigma_{{q}_{f}}(x_{k})=k$. Moreover, $x_{\aleph_{0}}=(10^{\infty})_{{q}_{2}}$ satisfies $\mathrm{card}\,\Sigma_{{q}_{f}}(x)=\aleph_{0}$.
\end{prop}

\begin{proof}
We proceed by induction. For $k=1$ we have $x_{1}=((10)^{\infty})_{{q}_{f}}$, therefore $\#\Sigma_{{q}_{f}}(x_{1})=1$ by Lemma~\ref{Unique expansions lemma}. Let us assume
$x_{k}=(1(0000)^{k-1}0(10)^{\infty})_{{q}_{f}}$ satisfies $\#\Sigma_{{q}_{f}}(x_{k})=k$, to prove our result it suffices to show that $x_{k+1}=(1(0000)^{k}0(10)^{\infty})_{{q}_{f}}$ satisfies $\#\Sigma_{{q}_{f}}(x_{k+1})=k+1$.

We begin by remarking that by Lemma~\ref{Unique expansions lemma} $((0000)^{k}0(10)^{\infty}))_{{q}_{f}}$ has a unique ${q_{f}}$-expansion, therefore there is a unique ${q_{f}}$-expansion of $x_{k+1}$ beginning with 1. Furthermore, it is a simple exercise to show that ${q}_{f}$ satisfies the equation $x^{4}=x^{3}+x^{2}+1$, which implies that  $(0(1101)(0000)^{k-1}0(10)^{\infty})$ is also a ${q}_f$-expansion for $x_{k+1}$.

To prove the claim, we will show that if $(\epsilon_{i})_{i=1}^{\infty}$ is a ${q}$-expansion for $x_{k+1}$ and $\epsilon_{1}=0$, then $\epsilon_{2}=1$, $\epsilon_{3}=1$ and $\epsilon_{4}=0$, which combined with our inductive hypothesis implies that the set of ${q}$-expansions for $x_{k+1}$ satisfying $\epsilon_{1}=0$ consists of $k$ distinct elements, combining these ${q}$-expansions with the unique ${q}$-expansion of $x_{k+1}$ satisfying $\epsilon_{1}=1$ we may conclude $\#\Sigma_{{q}_{f}}(x_{k+1})=k+1$.

Let us suppose $\epsilon_{1}=0$, if $\epsilon_{2}=0$; then we would require $$x_{k+1}=(1(0000)^{k}0(10)^{\infty})_{{q}_{f}}\leq (00(1)^{\infty})_{{q}_{f}},$$ however  $x_{k+1}>\frac{1}{{q}_{f}}$ and $\sum_{i=3}^{\infty}\frac{1}{{q}^{i}}<\frac{1}{q}$ for all ${q}>\frac{1+\sqrt{5}}{2}$, therefore $\epsilon_{2}=1$. If $\epsilon_{3}=0$ then we would require
\begin{equation}
\label{one}
x_{k+1}=(1(0000)^{k}0(10)^{\infty})_{{q}_{f}}\leq (010(1)^{\infty})_{{q}_{f}},
\end{equation}
which is equivalent to
$$
x_{k+1}=\frac{1}{{q}_{f}}+\frac{1}{{q}_{f}^{4k+3}} \sum_{i=0}^{\infty}\frac{1}{{q}_{f}^{2i}}\leq \frac{1}{{q}_{f}^{2}}+\frac{1}{{q}_{f}^{4}} \sum_{i=0}^{\infty}\frac{1}{{q}_{f}^{i}},
$$
however
$$
\frac{1}{{q}_{f}}=\frac{1}{{q}_{f}^{2}}+\frac{1}{{q}_{f}^{4}}
\sum_{i=0}^{\infty}\frac{1}{{q}_{f}^{i}},
$$
 whence (\ref{one}) cannot occur and $\epsilon_{3}=1$. Now let us suppose $\epsilon_{4}=1,$ then we must have
\begin{equation}
\label{two}
x_{k+1}=(1(0000)^{k}0(10)^{\infty})_{{q}_{f}}\geq (01110^{\infty})_{{q}_{f}},
\end{equation} which is equivalent to
\begin{equation}
\label{three}
x_{k+1}=\frac{1}{{q}_{f}}+\frac{1}{{q}_{f}^{4k+3}} \sum_{i=0}^{\infty}\frac{1}{{q}_{f}^{2i}}\geq \frac{1}{{q}_{f}^{2}}+\frac{1}{{q}_{f}^{3}}+\frac{1}{{q}_{f}^{4}}.
\end{equation}
The left hand side of (\ref{three}) is maximised when $k=1$, therefore to show that $\epsilon_{4}=0$ it suffices to show that
\begin{equation}
\label{four}
\frac{1}{{q}_{f}}+\frac{1}{{q}_{f}^{7}} \sum_{i=0}^{\infty}\frac{1}{{q}_{f}^{2i}}\geq \frac{1}{{q}_{f}^{2}}+\frac{1}{{q}_{f}^{3}}+\frac{1}{{q}_{f}^{4}}
\end{equation} does not hold. By a simple manipulation (\ref{four}) is equivalent to
\begin{equation}
\label{five}
{q}_{f}^{6}-{q}_{f}^{5}-2{q}_{f}^{4}+{q}_{f}^{2}+{q}_{f}+1\geq 0,
\end{equation}
but by an explicit calculation we can show that the left hand side of (\ref{five}) is strictly negative, therefore (\ref{two}) does not hold and $\epsilon_{4}=0$.

Now we consider $x_{\aleph_{0}}$, replicating our analysis for $x_{k}$ we can show that if $(\epsilon_{i})_{i=1}^{\infty}$ is a ${q}$-expansion for $x_{\aleph_{0}}$ and $\epsilon_{1}=0$ then $\epsilon_{2}=1$. Unlike our previous case it is possible for $\epsilon_{3}=0$, however in this case $\epsilon_{i}=1$ for all $i\geq 4$. If $\epsilon_{3}=1$, then as in our previous case we must have $\epsilon_{4}=0$. We observe that
$$
x_{\aleph_{0}}= (10^{\infty})_{{q}_{f}} = (010(1)^{\infty})_{{q}_{f}}= (011010^{\infty})_{{q}_{f}}.
$$
Clearly, there exists a unique ${q}$-expansion for $x_{\aleph_{0}}$ satisfying $\epsilon_{1}=1$ and a unique ${q}$-expansion for $x_{\aleph_{0}}$ satisfying $\epsilon_{1}=0$, $\epsilon_{2}=1$ and $\epsilon_{3}=0$. Therefore all other ${q}$-expansions of $x_{\aleph_{0}}$ have $(0110)$ as a prefix, repeating the above argument arbitrarily many times we can determine that all the ${q}_f$-expansions of $x_{\aleph_{0}}$ are of the form:
\begin{align*}
x_{\aleph_{0}}&=(10^{\infty})_{{q}_{f}}\\
&=(010(1)^{\infty})_{{q}_{f}}\\
&=(011010^{\infty})_{{q}_{f}}\\
&=(0110010(1)^{\infty})_{{q}_{f}}\\
&=(0110011010^{\infty})_{{q}_{f}}\\
&=(01100110010(1)^{\infty})_{{q}_{f}}\\
&=(01100110011010^{\infty})_{{q}_{f}},\\
& \vdots
\end{align*} which is clearly infinite countable.
\end{proof}

Thus, to prove Theorem~\ref{Main thm}, it suffices to show that $q_2\notin\B_3\cup\B_4$. This may look like a fairly innocuous exercise, but in reality it requires a substantial effort.

\section{Proof that ${q}_{2}\notin \B_{3}$}
\label{No 3 expansions}

By Lemma~\ref{Branching lemma dynamical} to show that ${q}_{2}\not\in \B_{k}$ for all $k\geq 3$ it suffices to show ${q}_{2}\notin \B_{3}$ and ${q}_{2}\notin \B_{4}$. To prove this we begin by characterising those $x\in S_{{q}_{2}}$ that satisfy $\#\Sigma_{{q}_{2}}(x)=2$. To simplify our notation, we denote for the rest of the paper $\beta:=q_2$ and $T_i:=T_{q_2,i}$ for $i=0,1$.

\begin{prop}
\label{two expansions corollary}
The only $x\in S_{{\beta}}$ that satisfy $\#\Sigma_{{\beta}}(x)=2$ are
$$
x=(01(10)^{\infty})_{{\beta}}=(10000(10)^{\infty})_{{\beta}} \textrm{ and }x=(0111(10)^{\infty})_{{\beta}}=(100(10)^{\infty})_{{\beta}}.
$$
\end{prop}
\begin{proof}It was shown in the proof of \cite[Proposition~2.4]{Sid1} that if $\frac{1+\sqrt5}2<q<q_f$ and $y, y+1$ have unique $q$-expansions, then necessarily $q=\beta$ and either $y=(0000(10)^\infty)_{\beta}$ and $y+1=(1(10)^\infty)_{\beta}$ or $y=(00(10)^\infty)_{\beta}$ and $y+1=(111(10)^\infty)_{\beta}$ respectively. Since for either case there exists a unique $x\in S_{\beta}$ such that $\beta x-1=y$, Lemma~\ref{Branching lemma dynamical} yields the claim.
\end{proof}

In what follows we shall let $(\epsilon^{1}_{i})_{i=1}^{\infty}=01(10)^{\infty}$, $(\epsilon_{i}^{2})_{i=1}^{\infty}=10000(10)^{\infty}$, $(\epsilon_{i}^{3})_{i=1}^{\infty}=0111(10)^{\infty}$ and $(\epsilon _{i}^{4})_{i=1}^{\infty}=100(10)^{\infty}$.

\begin{remark}Let $(\bar{\epsilon}_{i})_{i=1}^{\infty}=(1-\epsilon_{i})_{i=1}^{\infty}$, we refer to $(\bar{\epsilon}_{i})_{i=1}^{\infty}$ as the \textit{reflection} of $(\epsilon_{i})_{i=1}^{\infty}$. Clearly $(\bar{\epsilon}^{1}_{i})_{i=1}^{\infty}=(\epsilon_{i}^{4})_{i=1}^{\infty}$ and $(\bar{\epsilon}^{2}_{i})_{i=1}^{\infty}=(\epsilon_{i}^{3})_{i=1}^{\infty}$, this is to be expected as every $x\in I_{q}$ satisfies $\#\Sigma_{q}(x)=\#\Sigma_{q}(\frac{1}{{q}-1}-x)$ and mapping $(\epsilon_{i})_{i=1}^{\infty}$ to $(\bar{\epsilon}_{i})_{i=1}^{\infty}$ is a bijection between $\Sigma_{q}(x)$ and $\Sigma_{q}(\frac{1}{{q}-1}-x)$. If $(\epsilon^{1}_{i})_{i=1}^{\infty}$ and $(\epsilon^{2}_{i})_{i=1}^{\infty}$ were not the reflections of $(\epsilon^{4}_{i})_{i=1}^{\infty}$ and $(\epsilon^{3}_{i})_{i=1}^{\infty}$ respectively then there would exist other $x\in S_{{\beta}}$ satisfying $\#\Sigma_{{\beta}}(x)=2$, contradicting Proposition~\ref{two expansions corollary}.
\end{remark}

In this section we show that no $x\in I_{{\beta}}$ can satisfy $\#\Sigma_{{\beta}}(x)=3$. To show that ${\beta}\notin \B_{3}$ and ${\beta}\notin \B_{4}$ we will make use of the following proposition.

\begin{prop}
\label{unique transformation prop}
Suppose $x\in I_{{\beta}}$ satisfies $\#\Sigma_{{\beta}}(x)=2$ or equivalently $\#\Omega_{{\beta}}(x)=2$, then there exists a unique sequence of transformations $a$ such that $a(x)\in S_{{\beta}}$. Moreover, $a(x)=((\epsilon^{1}_{i})_{i=1}^{\infty})_{{\beta}}$ or $a(x)=((\epsilon^{3}_{i})_{i=1}^{\infty})_{{\beta}}$.
\end{prop}
\begin{proof}
As $\#\Omega_{{\beta}}(x)=2$ then there must exist $a$ satisfying $a(x)\in S_{{\beta}}$, otherwise $\#\Omega_{{\beta}}(x)=1$. We begin by showing uniqueness, suppose $a'$ satisfies $a'(x)\in S_{{\beta}}$ and $a'\neq a$. If $|a'|<|a|$ then either $a'$ is a prefix of $a$ in which case by (\ref{Cardinality inequality}) and Lemma~\ref{Branching lemma dynamical} we have that
$$
\#\Omega_{{\beta}}(x)\geq
\#\Omega_{{\beta}}(a'(x)) = \#\Omega_{{\beta}}(T_{0}(a'(x)))+
\#\Omega_{{\beta}}(T_{1}(a'(x)))\geq 3,
$$
which contradicts $\#\Omega_{{\beta}}(x)=2$. Alternatively if $a'$ is not a prefix of $a$ then there exists $b\in\bigcup_{n=0}^{\infty}\{T_{0},T_{1}\}^{n}$ such that $b(x)\in S_{{\beta}}$ and either $b0$ is a prefix for $a'$ and $b1$ is a prefix for $a$, or $b0$ is a prefix for $a$ and $b1$ is a prefix for $a'$. In either case it follows from (\ref{Cardinality inequality}) and Lemma~\ref{Branching lemma dynamical} that
$$
\#\Omega_{\beta}(x)\geq \#\Omega_{\beta}(b(x)) = \#\Omega_{\beta}(T_{0}(b(x)))+ \#\Omega_{\beta}(T_{1}(b(x)))\geq 4,
$$
a contradiction. By analogous arguments we can show that if $|a'|=|a|$ or $|a'|>|a|$ then this implies $\#\Omega_{\beta}(x)>2$, therefore $a$ must be unique.

Now let $a$ be the unique sequence of transformations such that $a(x)\in S_{{\beta}}$. By Lemma~\ref{Branching lemma dynamical},
$$
\#\Omega_{\beta}(T_{0}(a(x)))=\#\Omega_{\beta}(T_{1}(a(x)))=1.
$$
But it follows from Proposition~\ref{two expansions corollary} that this can only happen when $a(x)=((\epsilon^{1}_{i})_{i=1}^{\infty})_{{\beta}}$ or $a(x)=((\epsilon^{3}_{i})_{i=1}^{\infty})_{{\beta}}$.
\end{proof}

\begin{remark}
\label{Proof remark}
By Proposition~\ref{unique transformation prop}, to show that $x\in I_{{\beta}}$ satisfies $\mathrm{card}\,\Sigma_{{\beta}}(x)>2$ (or equivalently, $\mathrm{card}\,\Omega_{{\beta}}(x)>2$), it suffices to construct a sequence of transformations $a$ such that $a(x)\in S_{{\beta}}$ with $a(x)\neq ((\epsilon^{1}_{i})_{i=1}^{\infty})_{{\beta}}$ and $a(x)\neq ((\epsilon^{3}_{i})_{i=1}^{\infty})_{{\beta}}$. We will make regular use of this strategy in our later proofs.
\end{remark}

Before proving ${\beta}\notin \B_{3}$ it is appropriate to state numerical estimates\footnote{The explicit calculations performed in this paper were done using MATLAB. In our calculations we approximated ${\beta}$ by $1.710644095045033$, which is correct to the first fifteen decimal places.} for $S_{{\beta}}$, $((\epsilon^{1}_{i})_{i=1}^{\infty})_{{\beta}}$ and $((\epsilon^{3}_{i})_{i=1}^{\infty})_{{\beta}}$. Our calculations yield
$$
S_{\beta}=[0.584575\ldots, 0.822599\ldots],
$$
$$
((\epsilon^{1}_{i})_{i=1}^{\infty})_{{\beta}}= 0.645198\ldots \textrm{ and } ((\epsilon^{3}_{i})_{i=1}^{\infty})_{{\beta}}=0.761976\ldots.
$$ These estimates will make clear when $a(x)\in S_{{\beta}}$ and whether $a(x)=((\epsilon^{1}_{i})_{i=1}^{\infty})_{{\beta}}$ or $a(x)=((\epsilon^{3}_{i})_{i=1}^{\infty})_{{\beta}}$.

\begin{thm}
\label{Three expansions thm}
We have ${\beta}\notin \B_{3}$.
\end{thm}

\begin{proof}
Suppose $x'\in I_{{\beta}}$ satisfies $\#\Sigma_{{\beta}}(x')=3$ or equivalently $\#\Omega_{{\beta}}(x')=3$. Let $a$ denote the unique minimal sequence of transformations such that $a(x')\in S_{{\beta}}$. By considering reflections we may assume without loss in generality that
$$
\#\Omega_{{\beta}}(T_{1}(a(x')))=1 \textrm{ and } \#\Omega_{{\beta}}(T_{0}(a(x')))=2.
$$
Put $x=T_{1}(a(x'))$; by a simple argument it can be shown that $x\neq 0$, so we may assume that $x=(0^k(01)^{\infty})_{{\beta}}$ for some $k\geq 1$. To show that ${\beta}\notin \B_{3}$ we consider $T_{0}(a(x'))=x+1=(0^k(01)^{\infty})_{{\beta}}+1$, we will show that for each $k\geq 1$ there exists a finite sequence of transformations $a$ such that $a(x+1)\in S_{{\beta}}$, $a(x+1)\neq ((\epsilon_{i}^{1})_{i=1}^{\infty})_{{\beta}}$ and $a(x+1)\neq((\epsilon_{i}^{3})_{i=1}^{\infty})_{{\beta}}$. By Proposition~\ref{unique transformation prop} and Remark~\ref{Proof remark} this implies $\#\Omega_{{\beta}}(x+1)>2$, which is a contradiction and therefore ${\beta}\notin \B_{3}$.

Table~\ref{table2} states the orbits of $(0^k(01)^{\infty})_{{\beta}}+1$ under $T_{0}$ and $T_{1}$ until eventually $(0^k(01)^{\infty})_{{\beta}}+1$ is mapped into $S_{{\beta}}$. Table~\ref{table2} also includes the orbit of $1$ under $T_{0}$ and $T_{1}$ until $1$ is mapped into $S_{{\beta}}$. The reason we have included the orbit of $1$ is because $(0^k(01)^{\infty})_{{\beta}}+1\to 1$ as $k\to \infty$, therefore understanding the orbit of $1$ allows us to understand the orbit of $(0^k(01)^{\infty})_{{\beta}}+1$ for large values of $k$.
\begin{table}
\centering
\begin{tabular}{|c|c|}
\hline
$(0^k(01)^{\infty})_{{\beta}}+1$ & Iterates of $(0^k(01)^{\infty})_{{\beta}}+1$ (To $6$ decimal places)\\
\hline
$(0(01)^{\infty})_{{\beta}}+1$ & Unique ${q}$-expansion by Proposition~\ref{two expansions corollary}\\
$(00(01)^{\infty})_{{\beta}}+1$ & $1.177400, 1.014114, 0.734788$\\
$(000(01)^{\infty})_{{\beta}}+1$ & Unique ${q}$-expansion by Proposition~\ref{two expansions corollary}\\
$(0000(01)^{\infty})_{{\beta}}+1$ & $1.060622, 0.8143482$\\
$(00000(01)^{\infty})_{{\beta}}+1$ & $1.035438, 0.771266$\\
$(000000(01)^{\infty})_{{\beta}}+1$ & $1.020716, 0.746082$\\
$1$ & $1, 0.710644$\\
\hline
\end{tabular}
\medskip
\caption{Successive iterates of $(0^k(01)^{\infty})_{{\beta}}+1$ falling into $S_\beta\setminus\{(\e^1)_\beta, (\e^3)_\beta\}$}
\label{table2}
\end{table}

By inspection of Table~\ref{table2}, we conclude that for $1\leq k\leq 6$ either $(0^k(01)^{\infty})_{{\beta}}+1$ has a unique ${q}$-expansion which contradicts $\#\Omega_{{\beta}}(T_{0}(a(x')))=2$, or there exists $a$ such that $a((0^k(01)^{\infty})_{{\beta}}+1)\in S_{{\beta}}$ with $a((0^k(01)^{\infty})_{{\beta}}+1)\neq ((\epsilon^{1}_{i})_{i=1}^{\infty})_{{\beta}}$ and $a((0^k(01)^{\infty})_{{\beta}}+1)\neq ((\epsilon^{3}_{i})_{i=1}^{\infty})_{{\beta}}$, which contradicts $\#\Omega_{{\beta}}(x+1)=2$ by Proposition~\ref{unique transformation prop}. To conclude our proof, it suffices to show that for each $k\geq 7$ there exists $a$ such that $a((0^k(01)^{\infty})_{{\beta}}+1)\in S_{{\beta}}$, $a((0^k(01)^{\infty})_{{\beta}}+1)\neq ((\epsilon^{1}_{i})_{i=1}^{\infty})_{{\beta}}$ and $a((0^k(01)^{\infty})_{{\beta}}+1)\neq ((\epsilon^{3}_{i})_{i=1}^{\infty})_{{\beta}}$. For all $k\geq 7$ $(0^k(01)^{\infty})_{{\beta}}+1\in (1,(000000(01)^{\infty})_{{\beta}}+1)$, but by inspection of Table \ref{table2} it is clear that $T_{1}(x)\in (0.710644\ldots,0.746082\ldots)$ for all $x \in (1,(000000(01)^{\infty})_{{\beta}}+1)$.  Therefore we can infer that such an $a$ exists for all $k\geq 7$, which concludes our proof.
\end{proof}

\section{Proof that ${q_2}\notin \B_{4}$}

To prove ${\beta}\notin \B_{4}$ we will use a similar method to that used in the previous section, the primary difference being there are more cases to consider. Before giving our proof we give details of these cases.

Suppose $x'\in I_{\beta}$ satisfies $\#\Sigma_{\beta}(x')=4$ or equivalently $\#\Omega_{\beta}(x')=4$. Let $a'$ denote the unique minimal sequence of transformations such that $a'(x)\in S_{\beta}$. By Lemma~\ref{Branching lemma dynamical},
$$
\#\Omega_{\beta}(T_{0}(a'(x')))+ \#\Omega_{\beta}(T_{1}(a'(x')))=4.
$$
By Theorem~\ref{Three expansions thm}, $\#\Omega_{\beta}(T_{0}(a'(x')))\neq 3$ and $\#\Omega_{\beta}(T_{1}(a'(x')))\neq 3$, whence
\begin{equation}
\label{one ref}
\#\Omega_{\beta}(T_{0}(a'(x')))=
\#\Omega_{\beta}(T_{1}(a'(x')))=2.
\end{equation}
Letting $x=T_{1}(a'(x'))$, we observe that (\ref{one ref}) is equivalent to
\begin{equation}
\label{eq:sum2}
\#\Omega_{\beta}(x)=\#\Omega_{\beta}(x+1)=2.
\end{equation}

By Proposition~\ref{unique transformation prop}, there exists a unique sequence of transformations $a$ such that $a(x)\in S_{\beta}$ and $a(x)=((\epsilon^{1}_{i})_{i=1}^{\infty})_{\beta}$ or $a(x)=((\epsilon^{3}_{i})_{i=1}^{\infty})_{\beta}$. We now determine the possible unique sequences of transformations $a$ that satisfy $a(x)\in S_{\beta}$.

To determine the unique $a$ such that $a(x)\in S_{\beta}$, it is useful to consider the interval $[\frac{1}{\beta^{2}-1},\frac{\beta}{\beta^{2}-1}]$. The significance of this interval is that $T_{0}(\frac{1}{\beta^{2}-1})= \frac{\beta}{\beta^{2}-1}$ and $T_{1}(\frac{\beta}{\beta^{2}-1})=\frac{1}{\beta^{2}-1}$. The monotonicity of the maps $T_{0}$ and $T_{1}$ implies that if $x \in (0,\frac{1}{\beta-1})$ and $x\notin [\frac{1}{\beta^{2}-1},\frac{\beta}{\beta^{2}-1}]$, then there exists $i\in\{0,1\}$ and a minimal $k\geq 1$ such that $T_{i}^{k}(x)\in [\frac{1}{\beta^{2}-1},\frac{\beta}{\beta^{2}-1}]$. Furthermore, $S_{\beta}\subset [\frac{1}{\beta^{2}-1},\frac{\beta}{\beta^{2}-1}]$, in view of ${\beta}>\frac{1+\sqrt5}2$.

In particular, if $x\in (0,\frac{1}{\beta^{2}-1})$, then there exists a minimal $k\geq 1$ such that $T_{0}^{k}(x)\in(\frac{1}{\beta^{2}-1},\frac{\beta}{\beta^{2}-1});$ $T_{0}^{k}(x)$ cannot equal $\frac{1}{\beta^{2}-1}$ or $\frac{\beta}{\beta^{2}-1}$ as this would imply $\#\Omega_{\beta}(x)=1$. There are three cases to consider: either $T_{0}^{k}(x)\in S_{\beta}$, in which case $T_{0}^{k}(x)=((\epsilon^{1}_{i})_{i=1}^{\infty})_{\beta}$ or $T_{0}^{k}(x)=((\epsilon^{3}_{i})_{i=1}^{\infty})_{\beta}$ by Proposition~\ref{unique transformation prop}, or alternatively $T_{0}^{k}(x)\in (\frac{1}{\beta^{2}-1},\frac{1}{\beta})$ or $T_{0}^{k}(x)\in (\frac{1}{\beta(\beta-1)},\frac{\beta}{\beta^{2}-1})$. It is a simple exercise to show that if $T_{0}^{k}(x)=((\epsilon^{3}_{i})_{i=1}^{\infty})_{\beta}$  or $T_{0}^{k}(x)\in (\frac{1}{\beta(\beta-1)},\frac{\beta}{\beta^{2}-1})$ then $k\geq 2$. By Lemma~\ref{Bijection lemma} and Proposition~\ref{unique transformation prop}, if  $T_{0}^{k}(x)\in S_{\beta}$, then $$x=(0^{k}(\epsilon^{1}_{i})_{i=1}^{\infty})_{\beta} \textrm{ for some } k\geq 1 \textrm{ or } x=(0^{k}(\epsilon^{3}_{i})_{i=1}^{\infty})_{\beta} \textrm{ for some } k\geq 2.$$

For any ${q}\in (\frac{1+\sqrt{5}}{2},{q}_{f})$ and $y\in (\frac{1}{{q}^{2}-1},\frac{1}{q})$ there exists a unique minimal sequence $a''$ such that $a''(y)\in S_{q}$, moreover $a''(y)=(T_{{q},1}\circ T_{{q},0})^{j}(y)$ for some $j\geq 1$ and $(T_{{q},1}\circ T_{{q},0})^{i}(y)\in (\frac{1}{{q}^{2}-1},\frac{1}{q})$ for all $i<k$. For all $y\in (\frac{1}{{q}^{2}-1},\frac{1}{q})$ we have that $(T_{{q},1}\circ T_{{q},0})(y)=q^2y-1<{q} -1$; furthermore, it can be checked directly that $\beta -1 < ((\epsilon^{3}_{i})_{i=1}^{\infty})_{\beta}$, whence if $T_{0}^{k}(x)\in(\frac{1}{\beta^{2}-1},\frac{1}{\beta})$, then
$$
x=(0^{k}(01)^{j}(\epsilon^{1}_{i})_{i=1}^{\infty})_{\beta},
$$
for some $k\geq 1$ and $j\geq 1$. By a similar argument it can be shown that if $T_{0}^{k}(x)\in (\frac{1}{\beta(\beta-1)},\frac{\beta}{\beta^{2}-1})$, then $$x=(0^{k}(10)^{j}(\epsilon^{3}_{i})_{i=1}^{\infty})_{\beta},$$ for some $k\geq 2$ and $j\geq 1$. The above arguments are summarised in the following proposition.

\begin{prop}
\label{case prop}
Let $x$ be as in (\ref{eq:sum2}); then one of the following four cases holds:
\begin{align}
\label{Case 1}
x&=(0^{k}(\epsilon^{1}_{i})_{i=1}^{\infty})_{\beta} \textrm{ for some } k\geq 1,\\
\label{Case 2}
x&=(0^{k}(\epsilon^{3}_{i})_{i=1}^{\infty})_{\beta} \textrm{ for some } k\geq 2,\\
\label{Case 3}
x&=(0^{k}(01)^{j}(\epsilon^{1}_{i})_{i=1}^{\infty})_{\beta} \textrm{ for some } k\geq 1 \textrm{ and } j\geq 1
\end{align} or
\begin{equation}
\label{Case 4}
x=(0^{k}(10)^{j}(\epsilon^{3}_{i})_{i=1}^{\infty})_{\beta} \textrm{ for some } k\geq 2 \textrm{ and } j\geq 1.
\end{equation}
\end{prop}
To prove that $\beta\notin \B_{4}$ we will show that for each of the four cases described in Proposition~\ref{case prop} there exists $a$ such that
\begin{equation}
\label{eq:ax1}
a(x+1)\in S_{\beta}\setminus\{((\epsilon^{1}_{i})_{i=1}^{\infty})_{\beta}, ((\epsilon^{3}_{i})_{i=1}^{\infty})_{\beta}\},
\end{equation}
which contradicts $\#\Omega_{\beta}(x+1)=2$ by Proposition~\ref{unique transformation prop} and Remark~\ref{Proof remark}.

For the majority of our cases an argument analogous to that used in Section~\ref{No 3 expansions} will suffice, however in the case where $k=1,3$ in (\ref{Case 3}) and $k=2,4$ in (\ref{Case 4}) a different argument is required. We refer to these cases as the \textit{exceptional cases}. For the exceptional cases we will also show (\ref{eq:ax1}), however the approach used in slightly more technical and as such we will treat these cases separately.

\begin{prop}
\label{map prop}
For each of the cases described by Proposition~\ref{case prop} there exists $a$ such that (\ref{eq:ax1}) holds.
\end{prop}
\begin{proof}[Proof of Proposition~\ref{map prop} for the non-exceptional cases]
In the cases where $x=(0^{k}(\epsilon^{1}_{i})_{i=1}^{\infty})_{\beta}$ for some  $k\geq 1$ or $x=(0^{k}(\epsilon^{3}_{i})_{i=1}^{\infty})_{\beta}$ for some $k\geq 1$ it is clear that $x\to 0$ as $k\to \infty$, therefore to understand the orbit of  $(0^{k}(\epsilon^{1}_{i})_{i=1}^{\infty})_{\beta}+1$ or $(0^{k}(\epsilon^{3}_{i})_{i=1}^{\infty})_{\beta}+1$ for large values of $k$ it suffices to consider the orbit of $1$. Similarly, in the cases described by (\ref{Case 3}) and (\ref{Case 4}) if we fix $k\geq 1$ then as $j\to \infty$ both $(0^{k}(01)^{j}(\epsilon^{1}_{i})_{i=1}^{\infty})_{\beta}$ and
$(0^{k}(10)^{j}(\epsilon^{3}_{i})_{i=1}^{\infty})_{\beta}$ converge to $(0^{l}(10)^{\infty})$ for some $l\geq 1$, therefore to understand the orbits of $(0^{k}(01)^{j}(\epsilon^{1}_{i})_{i=1}^{\infty})_{\beta}+1$ and
$(0^{k}(10)^{j}(\epsilon^{3}_{i})_{i=1}^{\infty})_{\beta}+1$ for large values of $j$ it suffices to consider the orbit of $(0^{l}(10)^{\infty})_{\beta}+1$, for some $l\geq 1$. By considering these limits it will be clear when a sequence of transformations $a$ exists that satisfies (\ref{eq:ax1}) for large values of $k$ and $j$.

We begin by considering the case  $x=(0^{k}(\epsilon^{1}_{i})_{i=1}^{\infty})_{\beta}$. Table~\ref{table3} plots successive (unique) iterates of $(0^{k}(\epsilon^{1}_{i})_{i=1}^{\infty})_{\beta}+1$ until $(0^{k}(\epsilon^{1}_{i})_{i=1}^{\infty})_{\beta}+1$ is mapped into $S_{\beta}$ for $1\leq k\leq 6$. It is clear from inspection of Table \ref{table3} that for $1\leq k \leq 6$ there exists $a$ such that $a(x+1)\in S_{\beta}$, $a(x+1)\neq ((\epsilon^{1}_{i})_{i=1}^{\infty})_{\beta}$ and $a(x+1)\neq ((\epsilon^{3}_{i})_{i=1}^{\infty})_{\beta}$. The case where $k\geq 7$ follows from the fact that $(0^{k}(\epsilon^{1}_{i})_{i=1}^{\infty})_{\beta}+1 \in (1,(000000(\epsilon^{1}_{i})_{i=1}^{\infty})_{\beta}+1)$ for all $k\geq 7$ and  $T_{1}(y)\in (0.710644\ldots,0.754688\ldots)$ for all  $y\in(1,(000000(\epsilon^{1}_{i})_{i=1}^{\infty})_{\beta}+1)$. The case described by (\ref{Case 2}) follows by an analogous argument therefore the details are omitted, we just include the relevant orbits in Table \ref{table4}.

\begin{table}
\centering
\begin{tabular}{|c|c|}
\hline
$(0^{k}(\epsilon^{1}_{i})_{i=1}^{\infty})_{\beta}+1$ &  Iterates of $(0^{k}(\epsilon^{1}_{i})_{i=1}^{\infty})_{\beta}+1$ (to 6 decimal places)\\
\hline
$(0(\epsilon^{1}_{i})_{i=1}^{\infty})_{\beta}+1$ & $1.377166, 1.355842, 1.319363, 1.256961,$ \\
 & $1.150213, 0.967605,  0.655228$\\
$(00(\epsilon^{1}_{i})_{i=1}^{\infty})_{\beta}+1$ & $1.220482, 1.087810, 0.860857, 0.472620, 0.808484$\\
$(000(\epsilon^{1}_{i})_{i=1}^{\infty})_{\beta}+1$ & $1.128888, 0.931126, 0.592825$ \\
$(0000(\epsilon^{1}_{i})_{i=1}^{\infty})_{\beta}+1$ & $1.075344, 0.839532, 0.436141, 0.746082$\\
$(00000(\epsilon^{1}_{i})_{i=1}^{\infty})_{\beta}+1$ & $1.044044, 0.785989$\\
$(000000(\epsilon^{1}_{i})_{i=1}^{\infty})_{\beta}+1$ & $1.025747,0.754688$\\
$1$ & $1, 0.710644$\\
\hline
\end{tabular}
\caption{Successive iterates of $(0^{k}(\epsilon^{1}_{i})_{i=1}^{\infty})_{\beta}+1$}
\label{table3}
\end{table}
For the non-exceptional cases described by (\ref{Case 3}) and (\ref{Case 4}) an analogous argument works for the first few values of $k$ by considering the limit of $x+1$ as $j\to \infty$, therefore we just include the relevant orbits in Table \ref{table4}. It is clear by inspection of Table \ref{table4} that  $(0^{k}(01)^{j}(\epsilon^{1}_{i})_{i=1}^{\infty})_{\beta}+ 1\in(1,(00000001(\epsilon^{1}_{i})_{i=1}^{\infty})_{\beta}+1)$ for all $k\geq 7$ and $j\geq 1$, however $T_{1}(y)\in (0.710644\ldots,0.749023\ldots)$ for all $y\in (1,(00000001(\epsilon^{1}_{i})_{i=1}^{\infty})_{\beta}+1)$, therefore by inspection of Table \ref{table4} we can conclude the case described by (\ref{Case 3}) in the non-exceptional cases. Similarly, it is clear from inspection of Table \ref{table4} that $(0^{k}(10)^{j}(\epsilon^{3}_{i})_{i=1}^{\infty})_{\beta}+ 1\in(1,(0000000(10)^{\infty})_{\beta}+1)$ for all $k\geq 8$ and $j\geq 1$, however $T_{1}(y)\in (0.710644\ldots,0.7460826\ldots)$ for all $y\in(1,(0000000(10)^{\infty})_{\beta}+1)$, therefore by inspection of Table \ref{table4} we can conclude the case described by (\ref{Case 4}) in the non-exceptional cases.
\begin{table}
\centering
\begin{tabular}{|c|c|}
\hline
$(0^{k}(\epsilon^{3}_{i})_{i=1}^{\infty})_{\beta}+1$ &  Iterates of $(0^{k}(\epsilon^{3}_{i})_{i=1}^{\infty})_{\beta}+1$ (to 6 decimal places)\\
\hline
$(00(\epsilon^{3}_{i})_{i=1}^{\infty})_{\beta}+1$ & $1.260388, 1.156076, 0.977635, 0.672385$\\
$(000(\epsilon^{3}_{i})_{i=1}^{\infty})_{\beta}+1$ & $1.152216, 0.971032, 0.661091$ \\
$(0000(\epsilon^{3}_{i})_{i=1}^{\infty})_{\beta}+1$ & $1.088982, 0.862860, 0.476047, 0.814348$\\
$(00000(\epsilon^{3}_{i})_{i=1}^{\infty})_{\beta}+1$ & $1.052016, 0.799626$\\
$(000000(\epsilon^{3}_{i})_{i=1}^{\infty})_{\beta}+1$ & $1.030407, 0.762660$\\
$(0000000(\epsilon^{3}_{i})_{i=1}^{\infty})_{\beta}+1$ & $1.017775, 0.741051$\\
$1$ & $1, 0.710644$\\
\hline
$(00(01)^{j}(\epsilon^{1}_{i})_{i=1}^{\infty})_{\beta}+1$ &  Iterates of $(00(01)^{j}(\epsilon^{1}_{i})_{i=1}^{\infty})_{\beta}+1$\\
\hline
$(0001(\epsilon^{1}_{i})_{i=1}^{\infty})_{\beta}+1$ & $1.192123, 1.039298, 0.777869$\\
$(000101(\epsilon^{1}_{i})_{i=1}^{\infty})_{\beta}+1$ & $1.182431, 1.022720, 0.749510$ \\
$(00(01)^{\infty})_{\beta}+1$ & $1.177400, 1.014114, 0.734788$\\
\hline
$(0000(01)^{j}(\epsilon^{1}_{i})_{i=1}^{\infty})_{\beta}+1$ &  Iterates of $(0000(01)^{j}(\epsilon^{1}_{i})_{i=1}^{\infty})_{\beta}+1$\\
\hline
$(000001(\epsilon^{1}_{i})_{i=1}^{\infty})_{\beta}+1$ & $1.065653,0.822954, 0.407782, 0.697570$\\
$(00000101(\epsilon^{1}_{i})_{i=1}^{\infty})_{\beta}+1$ & $1.062342, 0.817289$ \\
$(0000(01)^{\infty})_{\beta}+1$ & $1.060622, 0.814348$\\
\hline
$(00000(01)^{j}(\epsilon^{1}_{i})_{i=1}^{\infty})_{\beta}+1$ &  Iterates of $(00000(01)^{j}(\epsilon^{1}_{i})_{i=1}^{\infty})_{\beta}+1$\\
\hline
$(0000001(\epsilon^{1}_{i})_{i=1}^{\infty})_{\beta}+1$ & $1.038379, 0.776297$\\
$(00000(01)^{\infty})_{\beta}+1$ & $1.035438, 0.771266$\\
\hline
$(000000(01)^{j}(\epsilon^{1}_{i})_{i=1}^{\infty})_{\beta}+1$ &  Iterates of $(000000(01)^{j}(\epsilon^{1}_{i})_{i=1}^{\infty})_{\beta}+1$\\
\hline
$(00000001(\epsilon^{1}_{i})_{i=1}^{\infty})_{\beta}+1$ & $1.022435, 0.749023$\\
$(000000(01)^{\infty})_{\beta}+1$ & $1.020716, 0.746082$\\
\hline
$(000(10)^{j}(\epsilon^{3}_{i})_{i=1}^{\infty})_{\beta}+1$ &  Iterates of $(000(10)^{j}(\epsilon^{3}_{i})_{i=1}^{\infty})_{\beta}+1$\\
\hline
$(00010(\epsilon^{3}_{i})_{i=1}^{\infty})_{\beta}+1$ & $1.168794,0.999391,0.709603$\\
$(000(10)^{\infty})_{\beta}+1$ & $1.177400, 1.014114, 0.734788$\\
\hline
$(00000(10)^{j}(\epsilon^{3}_{i})_{i=1}^{\infty})_{\beta}+1$ &  Iterates of $(00000(10)^{j}(\epsilon^{3}_{i})_{i=1}^{\infty})_{\beta}+1$\\
\hline
$(0000010(\epsilon^{3}_{i})_{i=1}^{\infty})_{\beta}+1$ & $1.057681,0.809317$\\
$(00000(10)^{\infty})_{\beta}+1$ & $1.060622, 0.814348$\\
\hline
$(000000(10)^{j}(\epsilon^{3}_{i})_{i=1}^{\infty})_{\beta}+1$ &  Iterates of $(000000(10)^{j}(\epsilon^{3}_{i})_{i=1}^{\infty})_{\beta}+1$\\
\hline
$(00000010(\epsilon^{3}_{i})_{i=1}^{\infty})_{\beta}+1$ & $1.033719,0.768326$\\
$(000000(10)^{\infty})_{\beta}+1$ & $1.035438, 0.771266$\\
\hline
$(0000000(10)^{j}(\epsilon^{3}_{i})_{i=1}^{\infty})_{\beta}+1$ &  Iterates of $(0000000(10)^{j}(\epsilon^{3}_{i})_{i=1}^{\infty})_{\beta}+1$\\
\hline
$(000000010(\epsilon^{3}_{i})_{i=1}^{\infty})_{\beta}+1$ & $1.019711,0.744363$\\
$(0000000(10)^{\infty})_{\beta}+1$ & $1.020716, 0.746082$\\
\hline
\end{tabular}
\medskip
\caption{Successive iterates of $(0^{k}(\epsilon^{3}_{i})_{i=1}^{\infty})_{\beta}+1$}
\label{table4}
\end{table}
\end{proof}

\begin{proof}[Proof of Proposition~\ref{map prop} for the exceptional cases] The reason we cannot use the same method as used for the non-exceptional cases is because as $j\to\infty$ the limits of $(0(01)^{j}(\epsilon^{1}_{i})_{i=1}^{\infty})_{\beta}+1, (000(01)^{j}(\epsilon^{1}_{i})_{i=1}^{\infty})_{\beta}+1, (00(10)^{j}(\epsilon^{3}_{i})_{i=1}^{\infty})_{\beta}+1$ and $(0000(10)^{j}(\epsilon^{3}_{i})_{i=1}^{\infty})_{\beta}+1$ all have unique ${\beta}$-expansions, which follows from Proposition~\ref{two expansions corollary}. As a consequence of the uniqueness of the ${\beta}$-expansion of the relevant limit, the number of transformations required to map $x+1$ into $S_{\beta}$ becomes arbitrarily large as $j\to \infty$. However, the following proposition shows that we can still construct an $a$ satisfying (\ref{eq:ax1}) for all but three of the exceptional cases.

\begin{prop}
\label{Constructed map prop}
The following identities hold:
\begin{equation}
\label{identity 1}
\begin{aligned}
((T_{1}\circ T_{0})^{j-2}\circ (T_{1})^{4})((0(01)^{j}(\epsilon^{1}_{i})_{i=1}^{\infty})_{\beta}+1)&= \frac{\beta-1}{\beta^{3}(\beta^{2}-1)}+\frac{1}{\beta^{2}-1}\\
&\approx 0.59282 \textrm{ for } j\geq 3,
\end{aligned}
\end{equation}

\begin{equation}
\label{identity 2}
\begin{aligned}
((T_{1}\circ T_{0})^{j}\circ (T_{1})^{2})((000(01)^{j}(\epsilon^{1}_{i})_{i=1}^{\infty})_{\beta}+1)&= \frac{\beta-1}{\beta^{3}(\beta^{2}-1)}+\frac{1}{\beta^{2}-1}\\
&\approx 0.59282 \textrm{ for } j\geq 1,
\end{aligned}
\end{equation}

\begin{equation}
\label{identity 3}
\begin{aligned}
((T_{0}\circ T_{1})^{j-1}\circ (T_{1})^{3})((00(10)^{j}(\epsilon^{3}_{i})_{i=1}^{\infty})_{\beta}+1)&= \frac{\beta}{\beta^{2}-1}+\frac{1-\beta}{\beta^{3}(\beta^{2}-1)}\\
&\approx 0.81434 \textrm{ for } j\geq 2
\end{aligned}
\end{equation}
and
\begin{equation}
\label{identity 4}
\begin{aligned}
((T_{0}\circ T_{1})^{j+1}\circ (T_{1}))((0000(10)^{j}(\epsilon^{3}_{i})_{i=1}^{\infty})_{\beta}+1)&= \frac{\beta}{\beta^{2}-1}+\frac{1-\beta}{\beta^{3}(\beta^{2}-1)}\\
&\approx 0.81434 \textrm{ for } j\geq 1.
\end{aligned}
\end{equation}
\end{prop}
\begin{proof}
Proving that each of the identities~(\ref{identity 1}), (\ref{identity 2}), (\ref{identity 3}) and (\ref{identity 4}) hold follow by similar arguments, we will therefore just show that (\ref{identity 1}) holds. Note that
$$
(0(01)^{j}(\epsilon^{1}_{i})_{i=1}^{\infty})_{\beta}+1= \frac{\beta^{2j+2}+\beta-1}{\beta^{2j+3}(\beta^{2}-1)}+1,
$$
for all $j\geq 1$. We observe the following:
\begin{align*}
&((T_{1}\circ T_{0})^{j-2}\circ (T_{1})^{4})\Big(\frac{\beta^{2j+2}+ \beta-1}{\beta^{2j+3}(\beta^{2}-1)}+1\Big)\\
=&(T_{1}\circ T_{0})^{j-2}\Big(\frac{\beta^{2j+2}+ \beta-1}{\beta^{2j-1}(\beta^{2}-1)}+ \beta^{4}-\beta^{3}-\beta^{2}-\beta-1\Big)\\
=&\beta^{2j-4}\Big(\frac{\beta^{2j+2}+\beta-1}{\beta^{2j-1}(\beta^{2}-1)}+ \beta^{4}-\beta^{3}-\beta^{2}-\beta-1\Big)- \sum_{i=0}^{j-3}\beta^{2i}\\
=&\frac{\beta^{2j+2}+\beta-1}{\beta^{3}(\beta^{2}-1)}+\beta^{2j}- \beta^{2j-1}-\beta^{2j-2}-\beta^{2j-3}-\beta^{2j-4}-
\frac{\beta^{2j-4}-1}{\beta^{2}-1}\\
=& \frac{\beta^{2j+2}}{\beta^{3}(\beta^{2}-1)}+\beta^{2j}-\beta^{2j-1}- \beta^{2j-2}-\beta^{2j-3}-\beta^{2j-4}-\frac{\beta^{2j-4}}{\beta^{2}-1}+ \frac{\beta-1}{\beta^{3}(\beta^{2}-1)}+ \frac{1}{\beta^{2}-1}.
\end{align*}Therefore, to conclude our proof, it suffices to show that
\begin{equation}
\label{last equation}
\frac{\beta^{2j+2}}{\beta^{3}(\beta^{2}-1)}+ \beta^{2j}-\beta^{2j-1}-\beta^{2j-2}-\beta^{2j-3}-\beta^{2j-4}- \frac{\beta^{2j-4}}{\beta^{2}-1}=0.
\end{equation}
Manipulating the left hand side of (\ref{last equation}) it is clear that satisfying (\ref{last equation}) is equivalent to
\[
 \frac{\beta^{2j-1}-\beta^{2j-4}+(\beta^{2j}-\beta^{2j-1}- \beta^{2j-2}-\beta^{2j-3}-\beta^{2j-4})(\beta^{2}-1)}{\beta^{2}-1}=0
 \]
 or
\[ \frac{\beta^{2j-3}(\beta-1)(\beta^{4}-2\beta^{2}-\beta-1)}{\beta^{2}-1}=0.
\]
This is true in view of $\beta^{4}-2\beta^{2}-\beta-1=0$.
\end{proof}

By Proposition~\ref{Constructed map prop} and Table~\ref{table5} which displays the orbits of the exceptional cases that are not covered by Proposition~\ref{Constructed map prop} we can conclude Proposition~\ref{map prop} for all the exceptional cases, therefore $\beta\notin \B_{4}$, and Theorem~\ref{Main thm} holds.

\begin{table}
\centering
\begin{tabular}{|c|c|}
\hline
Exceptional cases &  Iterates (to 6 decimal places) \\
\hline
$(001(\epsilon^{1}_{i})_{i=1}^{\infty})_{\beta}+1$ & $1.328654, 1.272854, 1.177400, 1.014114, 0.734788$\\
$(00101(\epsilon^{1}_{i})_{i=1}^{\infty})_{\beta}+1$ & $1.312076, 1.244495, 1.128888, 0.931126, 0.592825$\\
$(0010(\epsilon^{3}_{i})_{i=1}^{\infty})_{\beta}+1$ & $1.288747, 1.204588, 1.060622, 0.814348$\\
\hline
\end{tabular}
\medskip
\caption{Remaining exceptional cases: $k=1, j\in\{1,2\}$ in (\ref{Case 3}) and $k=2, j=1$ in (\ref{Case 4})}
\label{table5}
\end{table}

\end{proof}
\section{Open questions}
To conclude the paper, we pose a few open questions:
\begin{itemize}
\item What is the topology of $\B_{k}$ for $k\geq 2$? In particular, what is the smallest limit point of $\B_k$? Is it below or above the Komornik-Loreti constant introduced in \cite{KomLor}?
\item What is the smallest $q$ such that $x=1$ has $k$ $q$-expansions? (For $k=1$ this is precisely the Komornik-Loreti constant.)
\item What is the structure of $\B_{\aleph_0} \cap \bigl(\frac{1+\sqrt5}2, q_f\bigr)$? In view of the results of the present paper, knowing this would lead to a complete understanding of $\card\,\Sigma_q(x)$ for all $q\le q_f$ and all $x\in I_{q}$.
\item Let, as above,
$$
\B_{\infty}= \bigcap_{k=1}^{\infty} \B_{k} \cap \B_{\aleph_{0}}\cap \B_{2^{\aleph_{0}}}.
$$
By Theorem~\ref{all cardinalities}, $q_{f}$ is the smallest element of $\B_{\infty}.$ What is the second smallest element of $\B_{\infty}$? What is the topology of $\B_{\infty}$?
\item In \cite{AllClaSid} the authors study the order in which periodic orbits appear in the set of points with unique $q$-expansion, they show that as $q\uparrow 2$, the order in which periodic orbits appear in the set of uniqueness is intimately related to the classical Sharkovski\u{\i} ordering. Does a similar result hold in our case? That is, if $k> k'$ with respect to the usual Sharkovski\u{\i} ordering, does this imply $\B_{k}\subset \B_{k'}$?
\end{itemize}

\medskip\noindent {\bf Acknowledgment.} The authors are grateful to Vilmos Komornik for useful suggestions.


\begin{thebibliography}{100}

\bibitem{AllClaSid} J,-P. Allouche, M. Clarke and N. Sidorov, \textit{Periodic unique beta-expansions: the Sharkovski\u{\i} ordering}, Ergod. Th. Dynam. Systems {\bf 29} (2009), 1055--1074.
\bibitem{Baker} S. Baker, \textit{Generalised golden ratios over integer alphabets,} arXiv:1210.8397 [math.DS].
\bibitem{DaKa} Z. Dar\'{o}czy and I. Katai, \textit{Univoque sequences}, Publ. Math. Debrecen {\bf 42} (1993),  397--407.
\bibitem{EJ} P. Erd\H os, I. Jo\'o, {\em On the number of expansions $1=\sum q^{-n_i}$}, Ann. Univ. Sci. Budapest {\bf 35} (1992), 129--132.
\bibitem{Erdos} P. Erd\H{o}s, I. Jo\'{o} and V. Komornik, \textit{Characterization of the unique expansions $1 =\sum_{i=1}^{\infty}q^{-n_{i}}$ and
related problems}, Bull. Soc. Math. Fr. {\bf 118} (1990), 377--390.
\bibitem{GlenSid} P. Glendinning and N. Sidorov, \textit{Unique representations of real numbers in non-integer bases}, Math. Res. Letters {\bf 8} (2001), 535--543.
\bibitem{Kom}V. Komornik, Open problems session of ``Dynamical Aspects of Numeration'' workshop, Paris, 2006.
\bibitem{KomLor} V. Komornik and P. Loreti, Unique developments in non-integer bases, Amer. Math. Monthly 105 (1998), no. 7, 636--639.
\bibitem{Parry} W. Parry, \textit{On the $\beta$-expansions of real numbers}, Acta Math. Acad. Sci. Hung. {\bf 11} (1960) 401--416.
\bibitem{Renyi} A. R\'{e}nyi, \textit{Representations for real numbers and their ergodic properties}, Acta Math. Acad. Sci. Hung. {\bf 8} (1957) 477--493.
\bibitem{Sid} N. Sidorov, \textit{Almost every number has a continuum of $\beta$-expansions}, Amer. Math. Monthly {\bf 110} (2003), 838�-842.
\bibitem{Sid1} N. Sidorov, \textit{Expansions in non-integer bases: lower, middle and top orders}, J. Number Th. {\bf 129} (2009), 741--754.
\bibitem{Sid2} N. Sidorov, \textit{Universal $\beta$-expansions}, Period. Math. Hungar. {\bf 47} (2003), 221--231.
\bibitem{SidVer} N. Sidorov and A. Vershik, \textit{Ergodic properties of the Erd\H os measure, the entropy of the goldenshift, and related problems} Monatsh. Math. {\bf 126} (1998),  215--261.



\end{thebibliography}
\end{document}